\documentclass[12pt,a4paper]{amsart}
\usepackage{mathrsfs}
\usepackage{tikz}
\usepackage{amsfonts}
\usepackage{amsmath}
\usepackage{amssymb}
\usepackage{enumerate}
\usepackage{hyperref}
\usepackage{latexsym}
\usepackage{xcolor}
\usepackage{comment}
\usepackage[shortlabels]{enumitem}
\setlist[enumerate,1]{label={\upshape(\roman*)}}
\usepackage[numbers,sort&compress]{natbib}

\makeatletter

\newcommand{\Rmnum}[1]
{\expandafter\@slowromancap\romannumeral #1@}
\makeatother

\newtheorem{thm}{Theorem}[section]

\newtheorem{lemma}[thm]{Lemma}
\newtheorem{fact}[thm]{Fact}

\newcounter{foo}[subsection]
\newcounter{fooo}[section]
\newtheorem{step}[foo]{Step}
\newtheorem{stepp}[fooo]{Step}
\newtheorem{cor}[thm]{Corollary}

\newtheorem{example}[thm]{Example}
\newtheorem{defin}[thm]{Definition}

\theoremstyle{definition}
\newtheorem{remark}[thm]{Remark}

\title[Regular sets in Cayley sum graphs on generalized dicyclic groups]{Regular sets in Cayley sum graphs on generalized dicyclic groups}
\date{}

\thanks{*Corresponding author}
\author[Peng]{Meiqi Peng}
\address{School of Science\\China University of Geosciences\\Beijing 100083\\China}
\email{pengmq@email.cugb.edu.cn}

\author[Yang]{Yuefeng Yang*}
\address{School of Science\\China University of Geosciences\\Beijing 100083\\China}
\email{yangyf@cugb.edu.cn}

\begin{document}
	\begin{abstract}
		For a graph $\Gamma=(V(\Gamma),E(\Gamma))$, a subset $C$ of $V(\Gamma)$ is called an $(\alpha,\beta)$-regular set in $\Gamma$, if every vertex of $C$ is adjacent to exactly $\alpha$ vertices of $C$ and every vertex of $V(\Gamma)\setminus C$ is adjacent to exactly $\beta$ vertices of $C$. In particular, if $C$ is an $(\alpha,\beta)$-regular set in some Cayley sum graph of a finite group $G$ with connection set $S$, then $C$ is called an $(\alpha,\beta)$-regular set of $G$. In this paper, we consider a generalized dicyclic group $G$ and for each subgroup $H$ of $G$, by giving an appropriate connection set $S$, we determine each possibility for $(\alpha,\beta)$ such that $H$ is an $(\alpha,\beta)$-regular set of $G$.
	\end{abstract}
	
	
	
	\maketitle
	\section{Introduction}
	
	A {\em graph} $\Gamma$ is a pair $(V(\Gamma),E(\Gamma))$ of vertex set $V(\Gamma)$ and edge set $E(\Gamma)$, where $E(\Gamma)$ is a subset of the set of 2-element subsets of $V(\Gamma)$. Throughout this paper, all groups are assumed to be finite, and all graphs are finite, simple and undirected.
	
	Let $\Gamma=(V(\Gamma),E(\Gamma))$ be a graph. For two different vertices $x,y\in V(\Gamma)$, they are said to be {\em adjacent} if the set $\{x,y\}\in E(\Gamma)$. For non-negative integers $\alpha$ and $\beta$, a subset $C$ of $V(\Gamma)$ is called an {\em $(\alpha,\beta)$-regular set} \cite{DMC04} in $\Gamma$, if every vertex of $C$ is adjacent to exactly $\alpha$ vertices of $C$ and every vertex of $ V(\Gamma)\setminus C$ is adjacent to exactly $\beta$ vertices of $C$. In particular, a $(0,1)$-regular set in $\Gamma$ is called a {\em perfect code}, an {\em efficient dominating set} \cite{DIJ03,YPD14,DYP17,TWH98,MK12,YSK22} or an {\em independent perfect dominating set} \cite{TWH98,JL01,XM21}; a $(1,1)$-regular set in $\Gamma$ is called a {\em total perfect code}, or an {\em efficient open dominating set} \cite{CTT12,GH03,TWH98}. 
	
	Let $G$ be a group with the identity element $e$ and $S$ an {\em inverse-closed} subset of $G\setminus\{e\}$ (that is, $S^{-1}:=\{s^{-1}:s\in S\}=S$). The {\em Cayley graph}
	${\rm Cay}(G,S)$ on $G$ with {\em connection set} $S$ is the graph with vertex set $G$ and edge set $\{\{g, gs\}:g\in G,s\in S\}$.
	
	 Huang, Xia and Zhou \cite{HH18} first introduced the definition of a subgroup (total) perfect code of a group $G$ by using Cayley graphs on $G$. A subset of a group $G$ is called a \emph{\textup{(}total\textup{)} perfect code} of $G$ if it is a (total) perfect code in some Cayley graph of $G$. A (total) perfect code of $G$ is called a \emph{subgroup \textup{(}total\textup{)} perfect code} of $G$ if it is also a subgroup of $G$. Also in \cite{HH18}, the authors gave a necessary and sufficient condition for a normal subgroup of a group $G$ to be a subgroup (total) perfect code of $G$. Chen, Wang and Xia generalized this result on perfect codes to arbitrary subgroups \cite{JC20}. Ma, Walls, Wang and Zhou \cite{MXL03} proved that a group $G$ admits every subgroup as a perfect code if and only if $G$ has no elements of order 4. In \cite{JZ211,JZ212}, Zhang and Zhou gave a few necessary and sufficient conditions for a subgroup of a group to be a subgroup perfect code, and several results on subgroup perfect codes of metabelian groups, generalized dihedral groups, nilpotent groups and 2-groups. Further results on subgroup perfect codes in Cayley graphs, see for examples \cite{AB21,XCB25,KY22,YW23,JYZ23}.

	In \cite{WY22,WY23}, Wang, Xia and Zhou first introduced the concept of $(\alpha,\beta)$-regular set of a group $G$ by using Cayley graphs on $G$. A subset (resp. subgroup) of a group $G$ is called an {\em $(\alpha,\beta)$-regular set} (resp. a subgroup {\em $(\alpha,\beta)$-regular set}) of $G$ if it is an $(\alpha,\beta)$-regular set in some Cayley graph of $G$. In particular, a subgroup $(0, 1)$-regular set (resp. $(1,1)$-regular set) of $G$ is a subgroup perfect code (resp. subgroup total perfect code) of $G$. In \cite{XMW24}, Wang, Xu and Zhou determined when a non-trivial proper normal subgroup of a group is a $(0,\beta)$-regular set of the group and determined all subgroup $(0,\beta)$-regular sets of dihedral groups and dicyclic groups. In \cite{YK25}, Khaefi, Akhlaghi and Khosravi proved that if $H$ is a subgroup of $G$, then $H$ is an $(\alpha,\beta)$-regular set of $G$, for each $0\leq\alpha\leq|H|-1$ such that $\gcd(2,|H|-1)$ divides $\alpha$, and for each $0\leq\beta\leq|H|$ such that $\beta$ is even. Also, they proved that a subgroup $H$ of a group $G$ is a perfect code of $G$ if and only if it is an $(\alpha,\beta)$-regular set of $G$, for each $0\leq\alpha\leq|H|-1$ such that $\gcd(2,|H|-1)$ divides $\alpha$, and for each $0\leq\beta\leq|H|$. In addition, they showed that if $H$ is a subgroup of $G$, then $H$ is a perfect code of $G$ if and only if it is an $(\alpha,\beta)$-regular set of $G$ for each $0\leq\alpha\leq|H|-1$ such that $\gcd(2,|H|-1)$ divides $\alpha$, and for each $0\leq\beta\leq|H|$ such that $\beta$ is odd.

	The Cayley sum graph is first defined for abelian groups \cite{FRKC98} and then it is generalized to any arbitrary group in \cite{MA16}. Let $G$ be a group. An element $x\in G$ is called {\em square} if $x=y^2$ for some element $y\in G$. A subset $S$ of $G$ is called {\em square-free} if every element of $S$ is not square. We say that a subset $S$ of $G$ is {\em normal} if $g^{-1}Sg=\{g^{-1}sg:s\in S\}=S$ for every element $g\in G$. Let $S$ be a normal square-free subset of $G$. The {\em Cayley sum graph} of $G$ with the {\em connection set} $S$, denoted by ${\rm CayS}(G,S)$, is the graph with vertex set $G$ and two vertices $x$ and $y$ are adjacent whenever $xy\in S$. Note that the normality of $S$ implies that $xy\in S$ if and only if $yx\in S$. The square-free condition of $S$ ensures that ${\rm CayS}(G,S)$ has no loops, so that ${\rm CayS}(G,S)$ is a simple graph. Clearly, ${\rm CayS}(G,S)$ is an $|S|$-regular graph.
	
	In \cite{MXL01}, Ma, Feng and Wang studied the perfect codes in Cayley sum graphs, and defined a subgroup perfect code of a group by using Cayley sum graphs instead of Cayley graphs. More precisely, a subgroup of a group $G$ is said to be a {\em subgroup perfect code} of $G$ if the subgroup is a perfect code in some Cayley sum graph of $G$. Also in \cite{MXL01}, the authors reduced the problem of determining when a given subgroup of an abelian group is a perfect code to the case of abelian 2-groups, and classified the abelian groups whose all non-trivial subgroups are perfect codes. Ma, Wang and the second author \cite{MXL02} characterized all subgroup perfect codes of abelian groups.

	The total perfect codes in Cayley sum graphs have been studied, and a total perfect code of a group is also defined by using Cayley sum graphs instead of Cayley graphs.	A subset (resp. subgroup) $C$ of a group $G$ is called a {\em total perfect code} (resp. {\em subgroup total perfect code}) of $G$ if it is a total perfect code in some Cayley sum graph of $G$. Zhang \cite{ZJY24} gave some necessary conditions of a subgroup of a given group being a (total) perfect code in a Cayley sum graph of the group, and classified the Cayley sum graphs of some families of groups which admit a subgroup as a (total) perfect code. Wang, Wei, Xu and Zhou \cite{WXM24} gave two necessary and sufficient conditions for a subgroup of a group $G$ to be a total perfect code of $G$, and obtained two necessary and sufficient conditions for a subgroup of an abelian group $G$ to be a total perfect code of $G$. They also gave a classification of subgroup total perfect codes of a cyclic group, a dihedral group and a dicyclic group.
	
	Replacing Cayley sum graphs with Cayley graphs in the concept of a subgroup $(\alpha,\beta)$-regular sets of a group, the authors in \cite{XW24} obtained the concept of regular sets in a group. More precisely, a subset (resp. subgroup) $C$ of a group $G$ is called an {\em $(\alpha,\beta)$-regular set} (resp. a {\em subgroup $(\alpha,\beta)$-regular set}) of $G$ if it is an $(\alpha,\beta)$-regular set in some Cayley sum graph of $G$. In particular, a subgroup $(0, 1)$-regular set (resp. $(1,1)$-regular set) of $G$ is a subgroup perfect code (resp. subgroup total perfect code) of $G$. Also in \cite{XW24}, the authors obtained some necessary and sufficient conditions for a subgroup of a group $G$ to be a $(0,\beta)$-regular set of $G$, and characterized all possible subgroup $(0,\beta)$-regular sets of a cyclic group, a dihedral group and a dicyclic group. Seiedali, Khosravi and Akhlaghi \cite{FS24} gave a necessary and sufficient condition for a subgroup of an abelian group to be a subgroup $(\alpha,\beta)$-regular set. For each subgroup $H$ of a dihedral group $G$, by giving an appropriate connection set $S$, they determined each possibility for $(\alpha,\beta)$ such that $H$ is an $(\alpha,\beta)$-regular set of $G$.
	
	In this paper, we study regular sets in Cayley sum graphs on a generalized dicyclic group. For each subgroup $H$ of a generalized dicyclic group $G$, by giving an appropriate connection set $S$, we determine each possibility for $(\alpha,\beta)$ such that $H$ is an $(\alpha,\beta)$-regular set of $G$. In order to state our main results, we need additional notations and terminologies.
	
	Let $G$ be a group with the identity element $e$. For $g\in G$, let $o(g)$ denote the {\em order} of $g$, that is, the smallest positive integer $m$ such that $g^m=e$. An element $a$ is called an {\em involution} if $o(a)=2$. For a subgroup $H$ of $G$ and an element $a\in G$, $Ha=\{ha:h\in H\}$ (resp. $aH=\{ah:h\in H\}$) is called a {\em right coset} (resp. {\em left coset}) of $H$ in $G$. The {\em index} of a subgroup $H$ in $G$, denoted by $|G:H|$, is defined as the number of distinct right (or left) cosets of $H$ in $G$.
	
	In the remainder of this paper, $A$ always denotes an abelian group of even order with an involution $b^2$, and $G$ denotes the {\em generalized dicyclic group} generated by $A$ and $b$ where $bab^{-1}=a^{-1}$ for all $a\in A$, that is, $G=\langle A,b:b^2\in A,b^4=e,bab^{-1}=a^{-1}, a\in A\rangle$. Since $A$ is an abelian group, from the fundamental theorem of finitely generated abelian groups, we may assume
	\begin{align}
		A=\langle a_{1}\rangle \times \cdots\times \langle a_{\lambda}\rangle \times \langle a_{\lambda+1} \rangle \times \cdots \times\langle a_{\lambda+\mu} \rangle,\nonumber
	\end{align}
	with $o(a_i)=p_i^{e_i}$ for all $1\leq i\leq \lambda+\mu$, where $p_{i}$ is a prime, $e_{1}\leq \cdots\leq e_{\lambda}$, and $p_{j}=2$ if and only if $1\leq j\leq \lambda$. Let $a_0=b^2$ and $k$ be the maximum nonnegative integer such that $a_k$ is an involution. Let $\varphi_i$ be the projection from $A$ to $\langle a_i\rangle$ for $1\leq i\leq \lambda+\mu$. Denote
	\begin{align}\label{A'B}
		B=\langle a_{1}^{2}\rangle\times\cdots\times \langle a_{\lambda}^{2}\rangle \times \langle a_{\lambda+1}\rangle\times \cdots\times \langle a_{\lambda+\mu}\rangle.
	\end{align}
	
	Throughout this paper, we always assume that $H$ is a subgroup of $A$ and $z\in A$.
	Denote
	\begin{align}
		L_H&=\{a_i:\varphi_i(H)=\langle a_i\rangle,~1\leq i\leq\lambda\},~r_H=|\{i:\varphi_i(H)=\langle a_i\rangle,~1\leq i\leq k\}|,\nonumber\\
		m_H&=|\{i:\varphi_i(H)\neq\{e\},~1\leq i\leq\lambda\}|.\nonumber
	\end{align}
	If no confusion occurs, we write $t$ instead of $t_{H}$ for all symbols $t\in\{L,r,m\}$.
	
	 Recall that $A$ is an abelian group, $G=\langle A,b:b^2\in A,b^4=e,bab^{-1}=a^{-1}, a\in A\rangle$ and $H\leqslant A$. All subgroups of $G$ are of the form $H$ or $\langle H,zb\rangle$ with $|\langle H,zb\rangle:H|=2$, see Lemma \ref{lem1}. In the following theorem, by giving an appropriate connection set $S$, we determine each possibility for $(\alpha,\beta)$ such that $H$ is an $(\alpha,\beta)$-regular set of $G$.
	\begin{thm}\label{thm1}
		 The subgroup $H$ is an $(\alpha,\beta)$-regular set of $G$ for $(\alpha,\beta)\neq(0,0)$, if and only if $0\leq\alpha\leq (2^{|L|}-1)|H|/2^{|L|}-\epsilon$, $\beta=t|H|/2^{|L|}$ with $0\leq t\leq2^{|L|}-\varepsilon$, and one of the following occurs:
		\begin{enumerate}
			\item\label{thm1-1} $\epsilon=1$ and $\alpha$ is even, when $m=|L|=1$ and $b^2\in H\setminus B$;
			
			\item\label{thm1-2} $\epsilon=1$, when $m\geq|L|\geq1$, $m\neq1$ and $b^2\in H\setminus B$;
			
			\item\label{thm1-3} $\epsilon=0$ and $\alpha$ is even, when $r=0$ and $b^2\notin H\setminus B$;
			
			\item\label{thm1-4} $\epsilon=0$, when $r>0$ and $b^2\notin H\setminus B$.
		\end{enumerate}
		Here, $\varepsilon=0$ if $B\cup\{b^2\}\subseteq H$, $\varepsilon=2$ if $B\nleqslant H$, $b^2\notin H\cup B$ and $Hb^2\cap B\neq\emptyset$, and $\varepsilon=1$, otherwise.
	\end{thm}
	
	In the following theorem, by giving an appropriate connection set $S$, we determine each possibility for $(\alpha,\beta)$ such that $\langle H,zb\rangle$ is an $(\alpha,\beta)$-regular set of $G$.
	
	\begin{thm}\label{thm2}
		Without loss of generality we assume $|\langle H,zb\rangle:H|=2$. Then $\langle H,zb\rangle$ is an $(\alpha,\beta)$-regular set of $G$ for $(\alpha,\beta)\neq(0,0)$, if and only if $(\alpha,\beta)=(\eta+t'|H|/2^{|L|},\zeta+t|H|/2^{|L|})$ with $0\leq\eta\leq(2^{|L|}-1)|H|/2^{|L|}-\epsilon$, $0\leq\zeta\leq(2^{|L|}-\varepsilon)|H|/2^{|L|}$ and $0\leq t',t\leq 2^{|L|}$, and one of the following occurs:
		\begin{enumerate}
			\item\label{thm2-1} $\epsilon=1$ and $\beta$ is even, when $m>|L|$ and $b^2\in H\setminus B$;
			
			\item\label{thm2-2} $\epsilon=1$, when $m=|L|$ and $b^2\in H\setminus B$;
			
			\item\label{thm2-3} $\epsilon=0$, $\alpha$ and $\beta$ are even, when $m>|L|$, $r=0$ and $b^2\notin H\setminus B$;
			
			\item\label{thm2-4} $\epsilon=0$ and $\beta$ is even, when $m>|L|$, $r>0$ and $b^2\notin H\setminus B$;
			
			\item\label{thm2-5} $\epsilon=0$, when $m=|L|$ and $b^2\notin H\setminus B$.
			
		\end{enumerate}
		Here, $\varepsilon=0$ if $B\leqslant H$, and $(\varepsilon,t')=(1,t)$ if $B\nleqslant H$.
	\end{thm}
	
	The paper is organized as follows. In Section 2, we give some preliminary results which will be
	used in subsequent sections. In Section 3, by giving an appropriate connection set $S$, we determine each possibility for $\beta$ such that $H$ or $\langle H,zb\rangle$ is a $(0,\beta)$-regular set of $G$. In Section 4, we prove Theorems \ref{thm1} and \ref{thm2}.
	\section{preliminary}
	
	In this section, we give some basic results which will be used frequently in this paper.
	
	Let ${\rm Sq}(G)$ and ${\rm Nsq}(G)$ be the sets of all square and non-square elements of $G$, respectively. For a subgroup $K$ of $G$, denote $\mathcal{L}(K)=\min\{|{\rm Nsq}(G)\cap Kx|: x\in G\setminus K\}$.
	
	\begin{lemma}\label{lem3}
		Let $K$ be a subgroup of $G$ and $S$ be a normal and square-free subset of $G$. The following hold:
		\begin{enumerate}
			\item\label{lem3-1} The subgroup $K$ is an $(\alpha,\beta)$-regular set in ${\rm CayS}(G,S)$ if and only if $\alpha=|S\cap K|$ and $\beta=|S\cap Kx|$ for each $x\in G\setminus K$;
			
			\item\label{lem3-2} If $K$ is an $(\alpha,\beta)$-regular set of ${\rm CayS}(G,S)$, then $\alpha\leq|{\rm Nsq}(G)\cap K|$ and $\beta\leq \mathcal{L}(K)$.
		\end{enumerate}
	\end{lemma}
	\begin{proof}
		\ref{lem3-1} The necessity is immediate from \cite[Lemma 2.4 and Corollary 2.5]{FS24}. Next we prove the sufficiency. Since $\alpha=|S\cap K|$ and $\beta=|S\cap Kx|$ for each $x\in G\setminus K$, $S$ contains exactly $\alpha$ elements of $K$ and $\beta$ elements of $Kx$ for each $x\in G\setminus K$. Let $a\in G\setminus K$. Without loss of generality, we may assume $S\cap K=\{k_1,k_2,\ldots,k_{\alpha}\}$ and $S\cap Ka=\{a_1a,a_2a,\ldots,a_{\beta}a\}$ with $a_j\in K$ for $1\leq j\leq \beta$. For each $k\in K$, there are exactly $\alpha$ elements $k^{-1}k_1,k^{-1}k_2,\ldots,k^{-1}k_{\alpha}\in K$ such that $kk^{-1}k_{i}=k_i\in S$ for $1\leq i\leq \alpha$. For $a\in G\setminus K$, there are exactly $\beta$ elements $a_1,a_2,\ldots,a_{\beta}\in K$ such that $a_ja\in S$ for $1\leq j\leq\beta$. Since $a\in G\setminus K$ was arbitrary, every element of $K$ is adjacent to exactly $\alpha$ elements of $K$ in ${\rm CayS}(G,S)$, every element of $G\setminus K$ is adjacent to exactly $\beta$ elements of $K$ in ${\rm CayS}(G,S)$. Since $S$ is a normal and square-free subset of $G$, $K$ is an $(\alpha,\beta)$-regular set in ${\rm CayS}(G,S)$.
		
		\ref{lem3-2} This is immediate from \cite[Theorem 2.9]{FS24}.
	\end{proof}

	
	
	
	The following result determines all subgroups of a generalized dicyclic group.
	
	\begin{lemma}\label{lem1}
		All subgroups of $G$ are as follows:
		\begin{enumerate}
			\item\label{lem1-1} $A_1$, where $A_1\leqslant A$;
			
			\item\label{lem1-2} $\langle A_1,xb\rangle$, where $x\in A$ and $|\langle A_1,xb\rangle:A_1|=2$.
		\end{enumerate}
	\end{lemma}
	\begin{proof}
		Let $M$ be a subgroup of $G$. Since $|G:A|=2$, we have $G=A\cup Ab$, which implies $M=A_1\cup(M\cap Ab)$, where $A_1=M\cap A$. If $M\cap Ab=\emptyset$, then $A_1=M\leq A$.
		
		We only need to consider the case $M\cap Ab \neq \emptyset$. Let $xb\in M\cap Ab$ with $x\in A$. It follows that $b^2=(xb)^2\in M\cap A=A_1$, and so $|\langle A_1,xb\rangle:A_1|=2$.
		
		Pick $w\in M$. If $w\in A$, then $w\in A_1$, and so $w\in \langle A_1,xb\rangle$. Now suppose $w\in Ab$. Then there exists $a\in A$ such that $w=ab$. It follows that $ab=ax^{-1}xb$. Since $ab,xb,b^2\in M$, one gets $ax^{-1}=(ab)(xb)b^2\in M\cap A=A_1$, which implies $w=ab\in\langle A_{1},xb\rangle$. Since $w$ was arbitrary, we get $M\leqslant \langle A_{1},xb\rangle$.

		Let $a'(xb)^i\in \langle A_1,xb\rangle$ with $a'\in A_1$ and $i\in\{0,1,2,3\}$. If $i\in\{0,2\}$, then $a'(xb)^i\in A_1\subseteq M$ since $b^2\in A_1$. If $i=\{1,3\}$, then $a'(xb)^i=a'xb$ or $a'b^2xb$, which implies $a'(xb)^i\in M$ since $a',xb,b^2\in M$. Therefore, $\langle A_{1},xb\rangle\leqslant M$, and so $M=\langle A_{1},xb\rangle$.
	\end{proof}
	
	
	Recall that $H$ is a subgroup of $A$.
	\begin{lemma}\label{lem4}
		 The subgroup $H$ is an $(\alpha,\beta)$-regular set of $G$ if and only if $H$ is an $(\alpha,0)$-regular set of $G$ and a $(0,\beta)$-regular set of $G$.
	\end{lemma}
	\begin{proof}
		This is immediate from the fact that $H$ is a normal subgroup of $G$ and \cite[Lemma 2.6]{FS24}.
	\end{proof}
	
Recall $B=\langle a_{1}^{2}\rangle\times\cdots\times \langle a_{\lambda}^{2}\rangle \times \langle a_{\lambda+1}\rangle\times \cdots\times \langle a_{\lambda+\mu}\rangle$. The following fact gives all square elements in $G$.
	\begin{fact}\label{jb2} The set $B\cup\{b^2\}$ consists of all square elements in $G$, that is, ${\rm Sq}(G)=B\cup\{b^2\}$.
	\end{fact}
	
	Denote\begin{align}
		A'=\langle a_{1}^{2^{e_{1}-1}}\rangle \times \cdots \times\langle a_{\lambda}^{2^{e_{\lambda}-1}}\rangle.\label{A'}
	\end{align}
	The following fact determines all involutions of $G$.
	
	\begin{fact}\label{jb3}
		The set $A'\setminus\{e\}$ consists of all involutions of $G$.
	\end{fact}
	
	Two elements $a$ and $h$ of $G$ are said to be {\em conjugate} if there exists an element $g\in G$ such that $h=g^{-1}ag$. The {\em conjugacy class} of an element $a\in G$ is the set of conjugates of $a\in G$, which is denote by $a^{G}$. The following result classifies the conjugacy classes of $G$.
	
	\begin{lemma}\label{lem2}
		{{\rm (\cite[Lemma 2.4]{YJ25})}} The following hold:
		\begin{enumerate}
			\item\label{lem2-1} $a^{G}=\{a\}$, for each $a\in A'$;
			
			\item\label{lem2-2} $a^{G}=\{a,a^{-1}\}$, for each $a\in A\setminus A'$;
			
			\item\label{lem2-3} $(ab)^{G}=Bab$, for each $a\in A$.
			
		\end{enumerate}
	\end{lemma}
	
	
	\begin{cor}\label{cor1}
		Each inverse-closed subset of $A$ is normal.
	\end{cor}
	
	
	
	
	

Recall $L=\{a_i:\varphi_i(H)=\langle a_i\rangle,~1\leq i\leq\lambda\}$ and $m=|\{i:\varphi_i(H)\neq\{e\},~1\leq i\leq\lambda\}|$. In the following lemmas, we give some results on the subgroup $H$.

	\begin{lemma}\label{fact2}
		The following hold:
		\begin{enumerate}
			\item\label{fact2-1} $|H\cap B|=|H|/2^{|L|}$, $|H\cap A'|=2^m$ and $|H\cap A'\cap B|=2^{m-r}$;
			
			\item\label{fact2-2} If $m>|L|$, then $|H|/2^{|L|}$ is even.
		\end{enumerate}
	\end{lemma}
	\begin{proof}
		\ref{fact2-1} Since $m\geq |L|$, we may assume
		\begin{align}
			H=\langle a_{i_1}\rangle\times\cdots\times\langle a_{i_{|L|}}\rangle\times \langle a'_{i_{|L|+1}}\rangle\times\cdots\times\langle a'_{i_m}\rangle\times\langle a'_{\lambda+1}\rangle\times\cdots\times\langle a'_{\lambda+\mu}\rangle,\nonumber
		\end{align}
		where $1\leq i_1<i_2<\cdots<i_m\leq\lambda$, 
		$\{e\}\neq\langle a'_{h}\rangle\subsetneq\langle a_{h}\rangle$ with $i_{|L|+1}\leq h\leq i_m$ and $a'_{j}\in\langle a_{j}\rangle$ with $\lambda+1\leq j\leq\lambda+\mu$. In view of \eqref{A'B}, we get $H\cap B=\langle a_{i_1}^2\rangle\times\cdots\times\langle a_{i_{|L|}}^2\rangle\times \langle a'_{i_{|L|+1}}\rangle\times\cdots\times\langle a'_{i_m}\rangle\times\langle a'_{\lambda+1}\rangle\times\cdots\times\langle a'_{\lambda+\mu}\rangle$. It follows that $|H\cap B|=|H|/2^{|L|}$. Since $r=|\{i:\varphi_i(H)=\langle a_i\rangle,~1\leq i\leq k\}|$, from \eqref{A'}, one gets $H\cap A'=\langle a_{i_{1}}\rangle\times\cdots\times\langle a_{i_{r}}\rangle\times\langle a_{i_{r+1}}^{2^{e_{i_{r+1}}-1}}\rangle\times\cdots\times \langle a_{i_{m}}^{2^{e_{i_{m}}-1}}\rangle$ with $1\leq i_1<i_2<\cdots<i_r\leq k$, which implies $H\cap A'\cap B=\langle a_{i_{r+1}}^{2^{e_{i_{r+1}}-1}}\rangle\times\cdots\times\langle a_{i_{m}}^{2^{e_{i_{m}}-1}}\rangle$ by \eqref{A'B}. Then $|H\cap A'|=2^m$ and $|H\cap A'\cap B|=2^{m-r}$.
		
		\ref{fact2-2} By \ref{fact2-1}, we have $2^{m}\mid |H|$. The fact $m>|L|$ implies that $|H|/2^{|L|}$ is even.
	\end{proof}
	
	Let $C$ be a subset of $A$. Denote $a_C=e$ if $C=\emptyset$, and $a_C=\prod_{a\in C}a$ if $C\neq\emptyset$. Let $T$ be the set consisting of $a_i$ with $\varphi_i(H)\neq\langle a_i\rangle$ for $1\leq i\leq \lambda$. Throughout the paper, $\sqcup$ denotes the disjoint union. Then $L\sqcup T=\{a_1,\ldots,a_{\lambda}\}$.
	
	\begin{lemma}\label{ff3}
		If $L'\subseteq L$ and $T'\subseteq T$, then $|Ba_{L'}a_{T'}b\cap Ha_{T'}b|=|H|/2^{|L|}$.
	\end{lemma}
	\begin{proof}
		Since $a_{L'}\in H$, we have $|Ba_{L'}a_{T'}b\cap Ha_{T'}b|=|Ba_{L'}\cap H|=|B\cap H|$. This together with Lemma \ref{fact2} \ref{fact2-1} leads to $|Ba_{L'}a_{T'}b\cap Ha_{T'}b|=|H|/2^{|L|}$.
	\end{proof}

	\begin{lemma}\label{ff2}
		There are $2^{|L|}$ distinct $a_{L'}$ for $L'\subseteq L$ and $2^{\lambda-|L|}$ distinct $a_{T'}$ for $T'\subseteq T$.
	\end{lemma}
	\begin{proof}
		Let $L'\subseteq L$ and $T'\subseteq T$. For each element $a_i\in L$ and each symbol $K\in\{L',T'\}$, if $a_i\in K$, then $\varphi_i(a_{K})=a_i$; if $a_i\notin K$, then $\varphi_i(a_{K})=e$. It follows that $\varphi_i(a_{L'})\in\{a_i,e\}$ for each $a_i\in L$ and $\varphi_i(a_{T'})\in\{a_i,e\}$ for each $a_i\in T$. Note that $|T|=\lambda-|L|$. Since $L'\subseteq L$ and $T'\subseteq T$ were arbitrary, there are exactly  $2^{|L|}$ distinct $a_{L'}$ for $L'\subseteq L$ and $2^{\lambda-|L|}$ distinct $a_{T'}$ for $T'\subseteq T$.
	\end{proof}
	
	\begin{lemma}\label{f2}
		We have $\sqcup_{L'\subseteq L,T'\subseteq T}Ba_{L'}a_{T'}=A$ and $H=\sqcup_{L'\subseteq L}(H\cap B)a_{L'}$.
	\end{lemma}
	\begin{proof}
		In view of Lemma \ref{ff2}, there are $2^{|L|}$ distinct $a_{L'}$ for $L'\subseteq L$ and $2^{\lambda-|L|}$ distinct $a_{T'}$ for $T'\subseteq T$. By \eqref{A'B}, $Ba_K$ and $Ba_{K'}$ are pairwise disjoint for distinct subsets $K$ and $K'$ of $\{a_1,\ldots,a_{\lambda}\}$. In view of \eqref{A'B}, we get $|A:B|=2^{\lambda}$. The fact $L\sqcup T=\{a_1,\ldots,a_{\lambda}\}$ implies $\sqcup_{L'\subseteq L,T'\subseteq T}Ba_{L'}a_{T'}=A$. Since $L\subseteq H$, from Lemma \ref{fact2} \ref{fact2-1}, we get $H=\sqcup_{L'\subseteq L}(H\cap B)a_{L'}$.	\end{proof}

	\begin{lemma}\label{f1}
		Let $a\in A\setminus H$ and $m=|L|$. If $Ha=Ha^{-1}$, then there exists $h\in H$ such that $a^2=h^2$.
	\end{lemma}
	\begin{proof}
		Since $m=|L|$, we may assume
		\begin{align}
			H=\langle a_{i_1}\rangle\times\cdots\times\langle a_{i_{|L|}}\rangle\times\langle a_{\lambda+1}'\rangle\times\cdots\times\langle a_{\lambda+\mu}'\rangle\nonumber,
		\end{align}
		where $1\leq i_1<i_2<\cdots<i_{{|L|}}\leq\lambda$ and $a'_{j}\in\langle a_j\rangle$ with $\lambda+1\leq j\leq \lambda+\mu$. Since $Ha=Ha^{-1}$, we have $a^2\in H$. Since $a\in A$, from \eqref{A'B} and Fact \ref{jb2}, we obtain $a^2\in B$, and so $a^2\in H\cap B=\langle a_{i_1}^2\rangle\times\cdots\times\langle a_{i_{|L|}}^2\rangle\times\langle a_{\lambda+1}'\rangle\times\cdots\times\langle a_{\lambda+\mu}'\rangle$. It follows that there exists $h\in H$ such that $a^2=h^2$.
	\end{proof}

	\begin{lemma}\label{fact1}
		Let $a\in A\setminus H$. If $m=|L|$, then $Ha=Ha^{-1}$ if and only if $Ha\cap A'\neq\emptyset$.
	\end{lemma}
	\begin{proof}
		We first prove the necessity. Assume $Ha=Ha^{-1}$. In view of Lemma \ref{f1}, we have $a^2=h^2$ for some $h\in H$. It follows that $(h^{-1}a)^2=e$. By Fact \ref{jb3}, we get $h^{-1}a\in Ha\cap A'$, which implies $Ha\cap A'\neq\emptyset$.
		
		We next prove the sufficiency. Let $a'\in Ha\cap A'$. Then there exists $h\in H$ such that $a'=ha$. It follows that $a=h^{-1}a'$. Since $a'\in A'$, from Fact \ref{jb3}, one gets $a^2=(h^{-1}a')^2=h^{-2}a'^2=h^{-2}\in H$. Then $Ha=Ha^{-1}$.
	\end{proof}

	\begin{lemma}\label{fact3}
		Let $a\in A\setminus H$. Suppose $Ha=Ha^{-1}$. If $Ha\cap A'\neq\emptyset$ and $Ha\cap {\rm Sq}(G)\neq\emptyset$, then $Ha\cap A'\cap {\rm Sq}(G)\neq\emptyset$.
	\end{lemma}
	\begin{proof}
		In view of Fact \ref{jb2}, we obtain ${\rm Sq}(G)=B\cup\{b^2\}$, and so $Ha\cap {\rm Sq}(G)=Ha\cap (B\cup\{b^2\})\neq\emptyset$. If $b^2\in Ha$, then $Hb^2=Ha$, which implies $b^2\in Ha\cap A'\cap{\rm Sq}(G)$ from Fact \ref{jb3}.
		
		Now we only need to consider the case $b^2\notin Ha$. It follows that $Ha\cap B\neq\emptyset$. Since $a\notin H$, there exists $c\in B\setminus H$ such that $Hc=Ha$. Since $Ha=Ha^{-1}$, one gets $(Hc)^2=(Ha)^2=H$, which implies $Hc=Hc^{-1}$. Since $Ha=Hc$, we have $Hc\cap A'\neq\emptyset$, and so $(hc)^2=e$ for some $h\in H$ from Fact \ref{jb3}. It follows that $h^{2}=c^{-2}$. Since $h\in H$, from Lemma \ref{f2}, we have $h=c'a_{L'}$ for some $c'\in H\cap B$ and $L'\subseteq L$. It follows that $h^2={c'}^2a_{L'}^2=c^{-2}$, and so $(c'c)^2=a_{L'}^{-2}$. Since $c\in B\setminus H$ and $c'\in H\cap B$, we get $c'c\in B\setminus H$. It follows from Fact \ref{jb2} that $c'c\in Hc\cap {\rm Sq}(G)$. Since $Hc=Ha$, we obtain $c'c\in Ha\cap {\rm Sq}(G)$.
		
		By Fact \ref{jb3}, it suffices to show that $(c'c)^{2}=a_{L'}^{-2}=e$. The case $a_{L'}=e$ is trivial. Now suppose $a_{L'}\neq e$. Since $c'c\in (B\setminus H)\cap {\rm Sq}(G)$, from \eqref{A'B}, there exists $a'\in A\setminus H$ such that $c'c={a'}^2$. Pick $a_i\in L'$. Let $\varphi_i(a')=a_i^{t_i}$. Since $a'^4=(c'c)^2=a_{L'}^{-2}$ and $\varphi_i(a_{L'})=a_i$, we have ${a_i}^{4t_i}={a_i}^{-2}$, and so $4t_i\equiv-2~\pmod{2^{e_{i}}}$ with $e_i\geq 1$, which imply $2^{e_i-1}\mid 2t_i+1$. It follows that $e_i=1$, and so $a_i$ is an involution. Since $a_i\in L'$ was arbitrary, one gets $(c'c)^{2}=a_{L'}^{-2}=e$.
	\end{proof}

	\section{$(0,\beta)$-regular set}
	In this section, we give some results concerning $(0,\beta)$-regular sets in Cayley sum graphs on generalized dicyclic groups.
	\begin{lemma}\label{f3}
		Let $a\in A$ and $S$ be a normal and square-free subset of $G$. Suppose $a\in Ba_{L'}a_{T'}$ for some $L'\subseteq L$ and $T'\subseteq T$. If $S$ contains exactly $t$ sets of type $Ba_{L''}a_{T'}b$ with $L''\subseteq L$, then $|S\cap Hab|=t|H|/2^{|L|}$.
	\end{lemma}
	\begin{proof}
		Let $\mathscr{A}$ be the set consisting of $(L'',T'')$ such that $L''\subseteq L$, $T''\subseteq T$ and $S\cap Ba_{L''}a_{T''}b\neq\emptyset$. In view of Lemma \ref{lem2}, $Ba_{L''}a_{T''}b$ is normal for $(L'',T'')\in\mathscr{A}$. Since $S$ is normal, $(L'',T'')\in\mathscr{A}$ if and only if $Ba_{L''}a_{T''}b\subseteq S$ for $L''\subseteq L$ and $T''\subseteq T$. By Lemma \ref{f2}, we have $\cup_{(L'',T'')\in\mathscr{A}}Ba_{L''}a_{T''}b=S\cap Ab$.
		
		Since $a\in Ba_{L'}a_{T'}$, there exists $c\in B$ such that $a=ca_{L'}a_{T'}$. Since $a_{L'}\in H$ and $c\in B$, we get
		\begin{align}\label{SHab}
			|S\cap Hab|&=|S\cap Ab\cap Hca_{L'}a_{T'}b|\nonumber\\
			&=|S\cap Ab\cap Hca_{T'}b|\nonumber\\
			&=|(\cup_{(L'',T'')\in\mathscr{A}}Ba_{L''}a_{T''}b)\cap Hca_{T'}b|\nonumber\\
			&=|(\cup_{(L'',T'')\in\mathscr{A}}Ba_{L''}a_{T''}b)\cap Ha_{T'}b|.
		\end{align}
		By Lemma \ref{f2}, we get $H=\sqcup_{L''\subseteq L}(H\cap B)a_{L''}\subseteq \sqcup_{L''\subseteq L}Ba_{L''}$, and so $Ha_{T'}b\subseteq \sqcup_{L''\subseteq L}Ba_{L''}a_{T'}b$. In view of \eqref{A'B}, $Ba_K$ and $Ba_{K'}$ are pairwise disjoint for distinct subsets $K$ and $K'$ of $\{a_1,\ldots,a_{\lambda}\}$, which implies
		$(\cup_{(L'',T'')\in\mathscr{A}}Ba_{L''}a_{T''}b)\cap Ha_{T'}b=(\cup_{(L'',T')\in\mathscr{A}}Ba_{L''}a_{T'}b)\cap Ha_{T'}b$.
		
		By \eqref{SHab}, we obtain $|S\cap Hab|=|(\cup_{(L'',T')\in\mathscr{A}}Ba_{L''}a_{T'}b)\cap Ha_{T'}b|$. Since $(L'',T')\in\mathscr{A}$ if and only if $Ba_{L''}a_{T'}b\subseteq S$ for $L''\subseteq L$, from Lemmas \ref{ff3} and \ref{f2}, we get $|S\cap Hab|=t|H|/2^{|L|}$, where $S$ contains exactly $t$ sets of type $Ba_{L''}a_{T'}b$ with $L''\subseteq L$.
	\end{proof}
	
	In this section, let $J$ be the union of $Ha$ for $a\in A\setminus H$ satisfying $Ha=Ha^{-1}$ and $Ha\cap A'=\emptyset$.
	
	\begin{lemma}\label{f6}
		If $m>|L|$, then $J\neq\emptyset$.
	\end{lemma}
	\begin{proof}
		Since $m>|L|$, we may assume
		\begin{align}
			H=\langle a_{i_1}\rangle\times\cdots\times\langle a_{i_{|L|}}\rangle\times \langle a'_{i_{|L|+1}}\rangle\times\cdots\times\langle a'_{i_m}\rangle\times\langle a'_{\lambda+1}\rangle\times\cdots\times\langle a'_{\lambda+\mu}\rangle,\nonumber
		\end{align}
		where $1\leq i_1<i_2<\cdots<i_m\leq\lambda$, $\{e\}\neq\langle a'_{h}\rangle\subsetneq\langle a_{h}\rangle$ with $i_{|L|+1}\leq h\leq i_m$ and $a'_{j}\in\langle a_{j}\rangle$ with $\lambda+1\leq j\leq\lambda+\mu$. Let $l$ be the minimal positive integer such that $a_{i_{m}}^{l}\in H$. Since $o(a_{i_m})=2^{e_{i_m}}$ and $\{e\}\neq \langle a'_{i_m}\rangle\subsetneq\langle a_{i_m}\rangle$, one gets $l=2^{e'_{i_m}}$ with $1\leq e'_{i_m}< e_{i_m}$. It follows that $2\leq l<2^{e_{i_m}}$ and $a_{i_{m}}^{l/2}\in A\setminus H$. Then $Ha_{i_m}^{l/2}=Ha_{i_m}^{-l/2}$.
		
		It suffices to show that $Ha_{i_m}^{l/2}\cap A'=\emptyset$. Since $Ha_{i_m}^{l/2}\cap H=\emptyset$ and $\varphi_{i}(Ha_{i_m}^{-l/2})=\varphi_{i}(H)$ for $1\leq i\leq \lambda+\mu$ with $i\neq i_{m}$, we have $\varphi_{i_m}(Ha_{i_m}^{l/2})\cap \varphi_{i_m}(H)=\varphi_{i_m}(Ha_{i_m}^{l/2})\cap \langle a'_{i_m}\rangle=\emptyset$. By the minimality of $l$, one gets $\langle a_{i_m}^{2^{e_{i_m}-1}}\rangle\leqslant \langle a'_{i_m}\rangle=\langle a_{i_m}^{l}\rangle$, and so $\varphi_{i_m}(Ha_{i_m}^{l/2})\cap \langle a_{i_m}^{2^{e_{i_m}-1}}\rangle=\emptyset$, which imply $\varphi_{i_m}(Ha_{i_m}^{l/2})\cap A'=\emptyset$ from \eqref{A'}. Then $Ha_{i_m}^{l/2}\cap A'=\emptyset$.
	\end{proof}
	
	A {\em right transversal} (resp. {\em left transversal}) of $H$ in $A$ is defined as a subset of $A$ which contains exactly one element in each right coset (resp. left coset) of $H$ in $A$. Since $A$ is abelian, every right coset of any subgroup is also a left coset of the subgroup, for the sake of simplicity, we use the term ``transversal'' to substitute for ``right transversal'' or ``left transversal''.
	
	\begin{lemma}\label{fact4}
		There exists a transversal $I$ of $H$ in $A$ containing $e$ such that $I\setminus J$ is inverse-closed and $I$ contains a square element in each coset of $H$ having nonempty intersection with ${\rm Sq}(G)$.
	\end{lemma}
	\begin{proof}
		Let $I$ be a transversal of $H$ in $A$ containing $e$ such that $I$ contains a square element in each coset of $H$ having nonempty intersection with ${\rm Sq}(G)$. Note that $Ha\neq Ha^{-1}$, $Ha=Ha^{-1}$ and $Ha\cap A'\neq\emptyset$, or $Ha=Ha^{-1}$ and $Ha\cap A'=\emptyset$ for each $Ha\in A/H\setminus\{H\}$. For each $Ha\in A/H\setminus\{H\}$ satisfying $Ha\neq Ha^{-1}$, let $ha\in Ha$ and $h^{-1}a^{-1}\in Ha^{-1}$ belong to $I$ for some $h\in H$. For each $Ha\in A/H\setminus\{H\}$ satisfying $Ha=Ha^{-1}$ and $Ha\cap A'\neq\emptyset$, if $Ha\cap {\rm Sq}(G)=\emptyset$, then let $ha\in Ha\cap A'$ belong to $I$ for some $h\in H$; if $Ha\cap {\rm Sq}(G)\neq\emptyset$, from Lemma \ref{fact3}, then let $ha\in Ha\cap A'\cap {\rm Sq}(G)$ belong to $I$ for some $h\in H$. Since $J$ is the union of $Ha$ for $a\in A\setminus H$ satisfying $Ha=Ha^{-1}$ and $Ha\cap A'=\emptyset$, from Fact \ref{jb3}, $I\setminus J$ is inverse-closed.
	\end{proof}

	\subsection{$H$ is a $(0,\beta)$-regular set of $G$}
	
	In this subsection, by giving an appropriate connection set $S$, we determine each possibility for $\beta$ such that $H$ is a $(0,\beta)$-regular set of $G$.

	\begin{lemma}\label{lem7}
		The following hold:
		\begin{enumerate}
			\item\label{lem7-1} If $B\cup\{b^2\}\subseteq H$, then $\mathcal{L}(H)=|H|$;
			
			\item\label{lem7-4} If $B\leqslant H$ and $b^2\notin H$, then $\mathcal{L}(H)=|H|-1$;
			
			\item\label{lem7-3} If $B\nleqslant H$, $b^2\notin H\cup B$ and $Hb^2\cap B\neq\emptyset$, then $\mathcal{L}(H)=(2^{|L|}-1)|H|/2^{|L|}-1$;
			
			\item\label{lem7-2} If $B\nleqslant H$, and $b^2\in H\cup B$ or $Hb^2\cap B=\emptyset$, then $\mathcal{L}(H)=(2^{|L|}-1)|H|/2^{|L|}$.
			
		\end{enumerate}
	\end{lemma}
	\begin{proof}
		By Fact \ref{jb2}, one gets ${\rm Sq}(G)=B\cup\{b^2\}$. It follows that $Hx\subseteq {\rm Nsq}(G)$ for each $Hx\in G/H\setminus\{H\}$ with $Hx\cap(B\cup\{b^2\})=\emptyset$. For each $Hx\in G/H\setminus\{H\}$ with $Hx\cap B\neq\emptyset$ and $b^2\notin Hx$, there exists $c\in B$ such that $Hx=Hc$, which implies $|{\rm Nsq}(G)\cap Hx|=|{\rm Nsq}(G)\cap Hc|=|Hc\setminus (H\cap B)c|=(2^{|L|}-1)|H|/2^{|L|}$ from Lemma \ref{fact2} \ref{fact2-1}. To determine $\mathcal{L}(H)$, it suffices to consider the coset $Hb^2$.
		
		\ref{lem7-1} Since $B\cup\{b^2\}\subseteq H$, one gets $Hx\subseteq {\rm Nsq}(G)$ for each $x\in G\setminus H$, and so $\mathcal{L}(H)=|H|$.
		
		\ref{lem7-4} Since $B\leqslant H$, we have $Hx\cap B=\emptyset$ for each $Hx\in G/H\setminus\{H\}$. Since $B\leqslant H$ and $b^2\notin H$, one obtains ${\rm Nsq}(G)\cap Hb^2=Hb^2\setminus \{b^2\}$, and so $\mathcal{L}(H)=|H|-1$.
		
		\ref{lem7-3} Since $Hb^2\cap B\neq\emptyset$, one gets $Hb^2=Ha$ for some $a\in B$. Since $b^2\notin H\cup B$ and $Ba=B$, from Lemma \ref{fact2} \ref{fact2-1}, one gets $|{\rm Nsq}(G)\cap Hb^2|=|Ha\setminus(Ba\cup\{b^2\})|=|Ha\setminus (Ha\cap Ba)|-1=(2^{|L|}-1)|H|/2^{|L|}-1$. Then $\mathcal{L}(H)=(2^{|L|}-1)|H|/2^{|L|}-1$.
		
		\ref{lem7-2} If $b^2\in H$, then $\mathcal{L}(H)=(2^{|L|}-1)|H|/2^{|L|}$, and so \ref{lem7-2} is valid. Now suppose $b^2\notin H$. If $b^2\in B$, from Lemma \ref{fact2} \ref{fact2-1}, then $|{\rm Nsq}(G)\cap Hb^2|=|Hb^2\setminus B|=|Hb^2\setminus(H\cap B)b^2|=(2^{|L|}-1)|H|/2^{|L|}$, which implies that \ref{lem7-2} holds. If $b^2\notin B$, then $Hb^2\cap B=\emptyset$, and so $|{\rm Nsq}(G)\cap Hb^2|=|Hb^2\setminus \{b^2\}|=|H|-1$. Thus, \ref{lem7-2} holds.
	\end{proof}
	
	\begin{lemma}\label{lem13}
		Suppose that $I$ is a transversal of $H$ in $A$ containing $e$ such that $I$ contains a square element in each coset of $H$ having nonempty intersection with ${\rm Sq}(G)$. The following hold:
		\begin{enumerate}
			\item\label{lem13-1} If $B\cup\{b^2\}\subseteq H$, then $(I\setminus\{e\})x$ is square-free for each $x\in H$;
			
			\item\label{lem13-2} If $B\leqslant H$ and $b^2\notin H$, then $(I\setminus\{e\})x$ is square-free for each $x\in H\setminus B$;
			
			\item\label{lem13-3} If $B\nleqslant H$, $b^2\notin H\cup B$ and $Hb^2\cap B\neq\emptyset$, then $(I\setminus\{e\})x$ is square-free for each $x\in H\setminus (B\cup Bb^2)$;
			
			\item\label{lem13-4} If $B\nleqslant H$, and $b^2\in H\cup B$ or $Hb^2\cap B=\emptyset$, then $(I\setminus\{e\})x$ is square-free for each $x\in H\setminus B$.
		\end{enumerate}
	\end{lemma}
	
	\begin{proof}
		Note that $I\setminus\{e\}=((I\setminus\{e\})\cap (\cup_{a\in B\cup \{b^2\}}Ha))\cup (I\setminus(\cup_{a\in B\cup \{b^2\}}Ha))$. By Fact \ref{jb2}, one gets ${\rm Sq}(G)=B\cup \{b^2\}\subseteq \cup_{a\in B\cup \{b^2\}}Ha$. Since $Hah=Ha$ for each $h\in H$ with $a\in A$, one gets $(I\setminus(\cup_{a\in B\cup \{b^2\}}Ha))h\subseteq A\setminus(B\cup\{b^2\})$, which implies that $(I\setminus(\cup_{a\in B\cup \{b^2\}}Ha))h$ is square-free. It suffices to show that $((I\setminus\{e\})\cap (\cup_{a\in B\cup \{b^2\}}Ha))x$ is square-free for each case.
		
		\ref{lem13-1} Since $B\cup\{b^2\}\subseteq H$, we get $\cup_{a\in B\cup \{b^2\}}Ha=H$, and so $(I\setminus\{e\})\cap (\cup_{a\in B\cup \{b^2\}}Ha)=(I\setminus\{e\})\cap H=\emptyset$, which imply $(I\setminus\{e\})x=(I\setminus(\cup_{a\in B\cup \{b^2\}}Ha))x$ for each $x\in H$. Thus, \ref{lem13-1} is valid.
		
		\ref{lem13-2} Since $b^2\notin H$, one gets $Hb^2\cap H=\emptyset$. Since $B\leqslant H$, we have $Hb^2\cap B=\emptyset$ and $\cup_{a\in B\cup \{b^2\}}Ha=H\cup Hb^2$. Since $I$ contains a square element in each coset of $H$ having nonempty intersection with ${\rm Sq}(G)$ and $Hb^2\cap B=\emptyset$, from Fact \ref{jb2}, we have $b^2\in I$, which implies $(I\setminus\{e\})\cap (\cup_{a\in B\cup \{b^2\}}Ha)=(I\setminus\{e\})\cap (H\cup Hb^2)=(I\setminus\{e\})\cap Hb^2=\{b^2\}$. Then $((I\setminus\{e\})\cap (\cup_{a\in B\cup \{b^2\}}Ha))x=\{b^2x\}\subseteq Hb^2\setminus (B\cup Bb^2)\subseteq A\setminus (B\cup Bb^2)$ for each $x\in H\setminus B$, and so $((I\setminus\{e\})\cap (\cup_{a\in B\cup \{b^2\}}Ha))x$ is square-free. Thus, \ref{lem13-2} holds.
		
		\ref{lem13-3} Note that $I$ contains a square element in each coset of $H$ having nonempty intersection with ${\rm Sq}(G)$. Since $B\nleqslant H$, $b^2\notin H\cup B$ and $Hb^2\cap B\neq\emptyset$, from Fact \ref{jb2}, one obtains $(I\setminus\{e\})\cap (\cup_{a\in B\cup \{b^2\}}Ha)\subseteq B\cup \{b^2\}\subseteq B\cup Bb^2$, which implies $((I\setminus\{e\})\cap (\cup_{a\in B\cup \{b^2\}}Ha))x\subseteq (B\cup Bb^2)x\subseteq A\setminus (B\cup Bb^2)$ for each $x\in H\setminus (B\cup Bb^2)$. Then $((I\setminus\{e\})\cap (\cup_{a\in B\cup \{b^2\}}Ha))x$ is square-free for each $x\in H\setminus (B\cup Bb^2)$. Thus, \ref{lem13-3} is valid.
		
		\ref{lem13-4} Note that $I$ contains a square element in each coset of $H$ having nonempty intersection with ${\rm Sq}(G)$. If $b^2\in H$, from Fact \ref{jb2}, then $(I\setminus\{e\})\cap (\cup_{a\in B\cup \{b^2\}}Ha)=(I\setminus\{e\})\cap (\cup_{a\in B}Ha)\subseteq B\setminus\{e\}$, and so $((I\setminus\{e\})\cap (\cup_{a\in B\cup \{b^2\}}Ha))x\subseteq A\setminus (H\cup B)$ for each $x\in H\setminus B$, which imply that $((I\setminus\{e\})\cap (\cup_{a\in B\cup \{b^2\}}Ha))x$ is square-free. If $b^2\in B$, then $B\cup\{b^2\}=B$, and so $((I\setminus\{e\})\cap (\cup_{a\in B\cup \{b^2\}}Ha))x\subseteq Bx\subseteq A\setminus B$ for each $x\in H\setminus B$ from Fact \ref{jb2}, which imply that $((I\setminus\{e\})\cap (\cup_{a\in B\cup \{b^2\}}Ha))x$ is square-free.
		
		Now suppose $b^2\notin H\cup B$. It follows that $Hb^2\cap B=\emptyset$. Since $I$ contains a square element in each coset of $H$ having nonempty intersection with ${\rm Sq}(G)$, one has $I\cap Hb^2=\{b^2\}$. For each $x\in H\setminus B$, since $Hb^2\cap B=\emptyset$, we obtain $(I\cap Hb^2)x=\{b^2x\}\subseteq Hb^2\setminus Bb^2\subseteq A\setminus (B\cup \{b^2\})$ and $((I\setminus\{e\})\cap (\cup_{a\in B}Ha))x\subseteq Bx\subseteq A\setminus (B\cup Bb^2)$, which imply that $((I\setminus\{e\})\cap (\cup_{a\in B\cup \{b^2\}}Ha))x\subseteq A\setminus (B\cup \{b^2\})$. Then $((I\setminus\{e\})\cap (\cup_{a\in B\cup \{b^2\}}Ha))x$ is square-free for each $x\in H\setminus B$. Thus, \ref{lem13-4} is valid.
	\end{proof}

	\begin{lemma}\label{lem8}	
		The subgroup $H$ is a $(0,\beta)$-regular set in ${\rm CayS}(G,S)$ if and only if $\beta=t|H|/2^{|L|}$ and one of the following holds:
		\begin{enumerate}
			\item\label{lem8-1} $0\leq t\leq 2^{|L|}$, $B\cup\{b^2\}\subseteq H$;
			
			\item\label{lem8-2} $0\leq t\leq 2^{|L|}-1$, $B\leqslant H$ and $b^2\notin H$;
			
			\item\label{lem8-4} $0\leq t\leq 2^{|L|}-2$, $B\nleqslant H$, $b^2\notin H\cup B$ and $Hb^2\cap B\neq\emptyset$;
			
			\item\label{lem8-3} $0\leq t\leq 2^{|L|}-1$, $B\nleqslant H$, and $b^2\in H\cup B$ or $Hb^2\cap B=\emptyset$.
		\end{enumerate}
	\end{lemma}
	\begin{proof}
		
		We first prove the necessity. Pick an element $a\in A$. In view of Lemma \ref{f2}, one gets $a\in Ba_{L'}a_{T'}$ for some $L'\subseteq L$ and $T'\subseteq T$. Since $S$ is normal, from Lemma \ref{lem2}, we may assume that $S$ contains exactly $t$ sets of type $Ba_{L''}a_{T'}b$ with $L''\subseteq L$. In view of Lemmas \ref{lem3} and \ref{f3}, we get $\beta=|S\cap Hab|=t|H|/2^{|L|}$ with $0\leq t\leq \frac{\mathcal{L}(H)\cdot2^{|L|}}{|H|}$.
		
		If $B\cup\{b^2\}\subseteq H$, from Lemma \ref{lem7} \ref{lem7-1}, then $\mathcal{L}(H)=|H|$, and so $0\leq t\leq 2^{|L|}$, which imply that \ref{lem8-1} is valid. If $B\leqslant H$ and $b^2\notin H$, from Lemma \ref{lem7} \ref{lem7-4}, one gets $\mathcal{L}(H)=|H|-1$, and so $0\leq t\leq 2^{|L|}-1$, which imply that \ref{lem8-2} holds. If $B\nleqslant H$, $b^2\notin H\cup B$ and $Hb^2\cap B\neq\emptyset$, from Lemma \ref{lem7} \ref{lem7-3}, then $\mathcal{L}(H)=(2^{|L|}-1)|H|/2^{|L|}-1$, and so $0\leq t\leq 2^{|L|}-2$, which imply that \ref{lem8-4} is valid. If $B\nleqslant H$, and $b^2\in H\cup B$ or $Hb^2\cap B\neq\emptyset$, from Lemma \ref{lem7} \ref{lem7-2}, then $\mathcal{L}(H)=(2^{|L|}-1)|H|/2^{|L|}$, and so $0\leq t\leq 2^{|L|}-1$, which imply that \ref{lem8-3} holds.

		Next we prove the sufficiency. If $t=2^{|L|}$, then $\beta=|H|$, and so $S=G\setminus H$. Now we consider $0\leq t\leq 2^{|L|}-1$. By Lemma \ref{fact4}, let $I$ be a transversal of $H$ in $A$ containing $e$ such that $I\setminus J$ is inverse-closed and $I$ contains a square element in each coset of $H$ having nonempty intersection with ${\rm Sq}(G)$.
		
		In view of Lemma \ref{ff2}, there are $2^{|L|}-1$ distinct $a_{L'}$ for $\emptyset\neq L'\subseteq L\subseteq H$. By \eqref{A'B}, $Ba_K$ and $Ba_{K'}$ are pairwise disjoint for distinct subsets $K$ and $K'$ of $\{a_1,\ldots,a_{\lambda}\}$. If $b^2\notin H\cup B$ and $Hb^2\cap B\neq\emptyset$, then $H\cap Bb^2\neq\emptyset$, and there are $2^{|L|}-2$ distinct $a_{L'}$ for $\emptyset\neq L'\subseteq L$ and $a_{L'}\notin Bb^2$ from Lemma \ref{f2}. Since $0\leq t|H|/2^{|L|}\leq \mathcal{L}(H)$ from Lemma \ref{lem7}, each coset $Hx$ with $x\in G\setminus H$ has at least $t|H|/2^{|L|}$ non-square elements. Since $a_{L'}\in H$ with $L'\subseteq L$, from Lemma \ref{fact2} \ref{fact2-1}, one has $|H\cap Ba_{L'}|=|H|/2^{|L|}$.
		
		If $B\nleqslant H$, $b^2\notin H\cup B$ and $Hb^2\cap B\neq\emptyset$, then let $S$ be a union of $t|H|/2^{|L|+1}$ sets of type $\{ha',h^{-1}a'^{-1}\}$ with $ha'\in{\rm Nsq}(G)$ and $h\in H$ for each $Ha'\subseteq J$, $t$ sets of type $\cup_{x\in H\cap Ba_{L'}}(I\setminus (J\cup\{e\}))x$ and $t$ sets of type $\cup_{T'\subseteq T}Ba_{L'}a_{T'}b$ with $\emptyset\neq L'\subseteq L$ and $a_{L'}\notin Bb^2$. If $B\leqslant H$, or $b^2\in H\cup B$, or $Hb^2\cap B=\emptyset$, then let $S$ be a union of $t|H|/2^{|L|+1}$ sets of type $\{ha',h^{-1}a'^{-1}\}$ with $ha'\in{\rm Nsq}(G)$ and $h\in H$ for each $Ha'\subseteq J$, $t$ sets of type $\cup_{x\in H\cap Ba_{L'}}(I\setminus (J\cup\{e\}))x$ and $t$ sets of type $\cup_{T'\subseteq T}Ba_{L'}a_{T'}b$ with $\emptyset\neq L'\subseteq L$.
		
		
		We now claim that $H$ is a $(0,t|H|/2^{|L|})$-regular set in ${\rm CayS}(G,S)$. The proof proceeds in the following steps.
		\begin{step}\label{step-1}
			$S$ is normal and square-free.
		\end{step}
		In view of Fact \ref{jb2} and Lemma \ref{lem2}, $\cup_{T'\subseteq T}Ba_{L'}a_{T'}b$ with $\emptyset\neq L'\subseteq L$ is normal and square-free. By Lemma \ref{lem13}, $S$ is square-free.
		
		Let $L'$ be a nonempty subset of $L$ and $d\in H\cap Ba_{L'}$. Since $d\in Ba_{L'}$, there exists $c\in B$ such that $d=ca_{L'}$. By \eqref{A'B}, we have $a_{L'}^2\in B$, and so $d^{-1}=c^{-1}a_{L'}^{-1}\in H\cap Ba_{L'}^{-1}=H\cap Ba_{L'}$. Since $L'$ and $d$ were arbitrary, $H\cap Ba_{L'}$ is inverse-closed for each nonempty subset of $L$.
		
		
		
		Since $I\setminus (J\cup\{e\})$ is inverse-closed, $\cup_{x\in H\cap Ba_{L'}}(I\setminus (J\cup\{e\}))x$ for each nonempty subset $L'$ of $L$, and $\{ha',h^{-1}a'^{-1}\}$ with $ha'\in{\rm Nsq}(G)$ and $h\in H$ for each $Ha'\subseteq J$ are normal by Corollary \ref{cor1}, which imply that $S$ is normal.

		Thus, Step \ref{step-1} holds.
		\begin{step}\label{step-2}
			$|S\cap Hab|=t|H|/2^{|L|}$ for each $a\in A$.
		\end{step}
		Note that $S$ contains $t$ sets of type $Ba_{L'}a_{T'}b$ with $\emptyset\neq L'\subseteq L$ for each $T'\subseteq T$. For each $a\in A$, from Lemma \ref{f2}, one gets $a\in Ba_{L''}a_{T''}$ for some $L''\subseteq L$ and $T''\subseteq T$. By Lemma \ref{f3}, we get $|S\cap Hab|=t|H|/2^{|L|}$ for each $a\in A$.
		
		Thus, Step \ref{step-2} is valid.
		\begin{step}\label{step-3}
			$|S\cap Ha|=t|H|/2^{|L|}$ for each $a\in A\setminus H$.
		\end{step}
		
		Note that $|I\cap Ha|=1$ for each $a\in A\setminus H$. By Lemma \ref{fact2} \ref{fact2-1}, we have $|(\cup_{x\in H\cap Ba_{L'}}(I\setminus (J\cup\{e\}))x)\cap Ha|=|H\cap Ba_{L'}|=|H|/2^{|L|}$ with $\emptyset\neq L'\subseteq L\subseteq H$ for each $Ha$ satisfying $Ha\cap J\cap H=\emptyset$. Since $S$ contains $t$ sets of type $\cup_{x\in H\cap Ba_{L'}}(I\setminus (J\cup\{e\}))x$ with $\emptyset\neq L'\subseteq L$, from Lemma \ref{f2}, we get $|S\cap Ha|=t|H|/2^{|L|}$ for each $Ha$ satisfying $Ha\cap J\cap H=\emptyset$. If $J=\emptyset$, then the desired result follows. Now suppose $J\neq\emptyset$. It follows from Lemma \ref{fact1} that $m>|L|$. By Lemma \ref{fact2} \ref{fact2-2}, $|H|/2^{|L|}$ is even. Since $S$ contains $t|H|/2^{|L|+1}$ sets of type $\{ha',h^{-1}a'^{-1}\}$ with $ha'\in{\rm Nsq}(G)$ and $h\in H$ for each $Ha'\subseteq J$, one obtains $|S\cap Ha'|=t|H|/2^{|L|}$.
		
		Thus, Step \ref{step-3} holds. In view of Steps \ref{step-1}--\ref{step-3} and Lemma \ref{lem3} \ref{lem3-1}, our claim is valid.
	\end{proof}

	\subsection{$\langle H,zb\rangle$ is a $(0,\beta)$-regular set of $G$}
	In this subsection, we always assume that $\langle H,zb\rangle$ is a subgroup of $G$ such that $H$ has index $2$ of $\langle H,zb\rangle$.
	
	In the following lemma, by giving an appropriate connection set $S$ with $S\cap A=\emptyset$, we determine each possibility for $(\alpha,\beta)$ such that $\langle H,zb\rangle$ is an $(\alpha,\beta)$-regular set of $G$.
	
	\begin{lemma}\label{lem9}
		Let $S$ be a normal and square-free subset of $G$ with $S\cap A=\emptyset$. Then $\langle H,zb\rangle$ is an $(\alpha,\beta)$-regular set in ${\rm CayS}(G,S)$ if and only if one of the following holds:
		\begin{enumerate}
			\item\label{lem9-1} $B\leqslant H$, $\alpha=t'|H|/2^{|L|}$ and $\beta=t|H|/2^{|L|}$;
			
			\item\label{lem9-2} $B\nleqslant H$ and $\alpha=\beta=t|H|/2^{|L|}$.
		\end{enumerate}
		Here, $0\leq t,t'\leq 2^{|L|}$.
	\end{lemma}
	\begin{proof}
		Since $H$ has index $2$ of $\langle H,zb\rangle$, we have $\langle H,zb\rangle=H\cup Hzb$. Since $S\cap A=\emptyset$, we have
		\begin{align}\label{beta}
			|S\cap \langle H,zb\rangle a|=|S\cap (H\cup Hzb)a|=|S\cap Hzba|=|S\cap Hza^{-1}b|
		\end{align}
		for each $a\in A$. Since $z\in A$, from Lemma \ref{f2}, one gets $z\in Ba_{L'}a_{T'}$ for some $L'\subseteq L$ and $T'\subseteq T$.
		
		We first prove the necessity. Let $a\in A\setminus H$. Since $H$ has index $2$ of $\langle H,zb\rangle$, we have $\langle H,zb\rangle a\neq \langle H,zb\rangle$. In view of Lemma \ref{f2}, one gets $za^{-1}\in Ba_{L''}a_{T''}$ for some $L''\subseteq L$ and $T''\subseteq T$. Since $S$ is normal, from Lemma \ref{lem2}, we may assume that $S$ contains exactly $t$ sets of type $Ba_{L_0}a_{T''}b$ with $L_0\subseteq L$. By Lemma \ref{lem3} \ref{lem3-1}, Lemma \ref{f3} and \eqref{beta}, one has $\beta=|S\cap \langle H,zb\rangle a|= |S\cap Hza^{-1}b|=t|H|/2^{|L|}$ with $0\leq t\leq 2^{|L|}$.

		Suppose $B\leqslant H$. Note that $z\in Ba_{L'}a_{T'}$. Since $S$ is normal, from Lemma \ref{lem2}, we may assume that $S$ contains exactly $t'$ sets of type $Ba_{L''}a_{T'}b$ with $L''\subseteq L$. By Lemma \ref{lem3} \ref{lem3-1}, Lemma \ref{f3} and \eqref{beta}, we get $\alpha=|S\cap \langle H,zb\rangle|= |S\cap Hzb|=t'|H|/2^{|L|}$ with $0\leq t'\leq 2^{|L|}$. Thus, \ref{lem9-1} is valid.

		Suppose $B\nleqslant H$. It follows that $B\setminus H\neq\emptyset$. Since $z\in Ba_{L'}a_{T'}$, we have $zc^{-1}\in Ba_{L'}a_{T'}$ with $c\in B\setminus H$. By Lemma \ref{f3}, we have $|S\cap Hzb|=|S\cap Hzc^{-1}b|$. In view of Lemma \ref{lem3} \ref{lem3-1} and \eqref{beta}, one gets $\alpha=|S\cap \langle H,zb\rangle|=|S\cap Hzb|=|S\cap Hzc^{-1}b|=|S\cap \langle H,zb\rangle c|=\beta$. Thus, \ref{lem9-2} holds.
		
		We next prove the sufficiency. In view of Fact \ref{jb2} and Lemma \ref{lem2}, $Ba_{L''}a_{T''}b$ is normal and square-free for each $L''\subseteq L$ and $T''\subseteq T$.
		
		Suppose $B\leqslant H$. It follows from Lemma \ref{f2} that $H=\sqcup_{L''\subseteq L}Ba_{L''}$. Since $z\in Ba_{L'}a_{T'}\subseteq Ha_{L'}a_{T'}=Ha_{T'}$, one gets $Hz=Ha_{T'}=\cup_{L''\subseteq L}Ba_{L''}a_{T'}$, which implies $(\cup_{L''\subseteq L}Ba_{L''}a_{T'})\cap (A\setminus Hz)=\emptyset$. Let $S$ be a union of $t'$ sets of type $Ba_{L''}a_{T'}b$ with $L''\subseteq L$, and $t$ sets of type $Ba_{L''}a_{T''}b$ with $L''\subseteq L$ for each $T''\subseteq T$ with $(\cup_{L''\subseteq L}Ba_{L''}a_{T''})\cap (A\setminus Hz)\neq\emptyset$. Since $z\in Ba_{L'}a_{T'}$, from \eqref{beta} and Lemma \ref{f3}, we obtain $|S\cap \langle H,zb\rangle|=|S\cap Hzb|=t'|H|/2^{|L|}$. For each $x\in A\setminus H$, since $zx^{-1}\in A$, from Lemma \ref{f2}, one gets $zx^{-1}\in Ba_{L''}a_{T''}$ for some $L''\subseteq L$ and $T''\subseteq T$, which implies $zx^{-1}\in Ba_{L''}a_{T''}\cap (A\setminus Hz)$. In view of \eqref{beta} and Lemma \ref{f3}, one gets $|S\cap \langle H,zb\rangle x|=|S\cap Hzx^{-1}b|=t|H|/2^{|L|}$ for each $x\in A\setminus H$. By Lemma \ref{lem3} \ref{lem3-1}, $\langle H,zb\rangle$ is a $(t'|H|/2^{|L|},t|H|/2^{|L|})$-regular set in ${\rm CayS}(G,S)$ with $0\leq t,t'\leq 2^{|L|}$.

		Suppose $B\nleqslant H$. Let $S$ be a union of $t$ sets of type $\cup_{T''\subseteq T}Ba_{L''}a_{T''}b$ with $L''\subseteq L$. For each $za^{-1}\in A$, from Lemma \ref{f2}, one gets $za^{-1}\in Ba_{L''}a_{T''}$ for some $L''\subseteq L$ and $T''\subseteq T$. By \eqref{beta} and Lemma \ref{f3}, one gets $|S\cap \langle H,zb\rangle a|=|S\cap Hza^{-1}b|=t|H|/2^{|L|}$ for each $a\in A$. In view of Lemma \ref{lem3} \ref{lem3-1}, $\langle H,zb\rangle$ is a $(t|H|/2^{|L|},t|H|/2^{|L|})$-regular set in ${\rm CayS}(G,S)$ with $0\leq t\leq 2^{|L|}$.
	\end{proof}

	In the following lemma, by giving an appropriate connection set $S$ with $S\subseteq A$, we determine each possibility for $\beta$ such that $\langle H,zb\rangle$ is a $(0,\beta)$-regular set of $G$.
	
	Since $b^2=(zb)^2$, one gets $b^2\in H$.

	\begin{lemma}\label{lem18}
		Let $S$ be a normal and square-free subset of $G$ with $S\subseteq A$. Then $\langle H,zb\rangle$ is a $(0,\beta)$-regular set in ${\rm CayS}(G,S)$ if and only if one of the following holds:
		\begin{enumerate}
			\item\label{lem18-1} $B\leqslant H$ and $0\leq\beta\leq |H|$;
			
			\item\label{lem18-2} $B\nleqslant H$ and $0\leq\beta\leq (2^{|L|}-1)|H|/2^{|L|}$.
		\end{enumerate}
		Moreover, if $m>|L|$, then $2\mid\beta$.
	\end{lemma}
	
	\begin{proof}
		Since $H$ has index $2$ of $\langle H,zb\rangle$ and $S\subseteq A$, one gets
		\begin{align}\label{Aemptyset}
			S\cap \langle H,zb\rangle x=S\cap (H\cup Hzb)x=S\cap Hx
		\end{align}
		for each $x\in A\setminus H$.
		
		We first prove the necessity. Since $|G:A|=2$, one has $G=A\cup Ab$. By Fact \ref{jb2}, we have $|{\rm Nsq}(G)\cap Hab|=|H|$ for each $a\in A$. It follows that $\mathcal{L}(H)=\min\{|{\rm Nsq}(G)\cap Hx|:x\in G\setminus H\}=\min\{|{\rm Nsq}(G)\cap Hx|:x\in A\setminus H\}$. Let $\mathcal{L}(H)=|{\rm Nsq}(G)\cap Hx_0|$ with $x_0\in A\setminus H$. Since $S$ is square-free, one gets $S\cap Hx_0\subseteq {\rm Nsq}(G)\cap Hx_0$. By Lemma \ref{lem3} \ref{lem3-1} and \eqref{Aemptyset}, one obtains $\beta=|S\cap \langle H,zb\rangle x|=|S\cap Hx|$ for all $x\in A\setminus H$. It follows that $\beta=|S\cap Hx_0|\leq |{\rm Nsq}(G)\cap Hx_0|=\mathcal{L}(H)$.
		
		Suppose $B\leqslant H$. Since $b^2\in H$, from Lemma \ref{lem7} \ref{lem7-1}, we have $\beta\leq \mathcal{L}(H)=|H|$ for each $x\in A\setminus H$. Thus, \ref{lem18-1} is valid.
		
		Suppose $B\nleqslant H$. Since $b^2\in H$, from Lemma \ref{lem7} \ref{lem7-2}, one gets $\beta\leq \mathcal{L}(H)=(2^{|L|}-1)|H|/2^{|L|}$ for each $x\in A\setminus H$. Thus, \ref{lem18-2} holds.

		Now suppose $m>|L|$. In view of Lemma \ref{f6}, one has $J\neq\emptyset$. By Lemma \ref{lem3} \ref{lem3-1} and \eqref{Aemptyset}, we get $\beta=|S\cap Ha|$ for $Ha\subseteq J$. For each $ha\in S\cap Ha$ with $Ha\subseteq J$ and $h\in H$, since $S$ is normal and $Ha\cap A'=\emptyset$, we have $h^{-1}a^{-1}\in S\cap Ha^{-1}=S\cap Ha$ from Lemma \ref{lem2}. Since $Ha\cap A'=\emptyset$, from Fact \ref{jb3}, one obtains $2\mid|S\cap Ha|$, and so $2\mid\beta$. The second statement follows.
		
		Next we prove the sufficiency. By Lemma \ref{fact4}, let $I$ be a transversal of $H$ in $A$ containing $e$ such that $I\setminus J$ is inverse-closed and $I$ contains a square element in each coset of $H$ having nonempty intersection with ${\rm Sq}(G)$.
		
		By Lemma \ref{fact2} \ref{fact2-1}, we have $|H\setminus B|=(2^{|L|}-1)|H|/2^{|L|}$, $|H\cap A'|=2^m$ and $|(H\cap A')\setminus B|=2^m-2^{m-r}$. Note that $b^2\in H$. It follows that $\langle b^2\rangle\leqslant H$, and so $|H\cap A'|=2^m\geq2$, which imply $m\geq 1$. By Lemma \ref{lem7} \ref{lem7-1} and \ref{lem7-2}, we have $\beta\leq \mathcal{L}(H)$, which implies that each coset $\langle H,zb\rangle x$ with $x\in A\setminus H$ has at least $\beta$ non-square elements in $A$ by \eqref{Aemptyset}. Let $\sigma(n)=\frac{1-(-1)^{n}}{2}$ for an integer $n$. Now we give an appropriate connection set $S$ for each case respectively.
		\begin{itemize}[leftmargin=*]
			\item Suppose $B\leqslant H$, $m=|L|$ and $\beta\leq 2^m$. Let $S$ be a union of $\beta$ sets of type $(I\setminus\{e\})a'$ with $a'\in H\cap A'$.
			
			\item Suppose $B\leqslant H$, $m=|L|$ and $\beta> 2^m$. Let $S$ be a union of $2^m-\sigma(\beta)$ sets of type $(I\setminus\{e\})a'$ with $a'\in H\cap A'$ and $\frac{\beta-(2^m-\sigma(\beta))}{2}$ sets of type $(I\setminus\{e\})a\cup(I\setminus\{e\})a^{-1}$ with $a\in H\setminus A'$.
			
			\item Suppose $B\nleqslant H$, $m=|L|$ and $\beta\leq 2^m-2^{m-r}$. Let $S$ be a union of $\beta$ sets of type $(I\setminus \{e\})c'$ with $c'\in (H\cap A')\setminus B$.
			
			\item Suppose $B\nleqslant H$, $m=|L|$ and $\beta>2^m-2^{m-r}$. Let $S$ be a union of $2^m-2^{m-r}-\sigma(\beta)+(-1)^{\beta+1}\cdot\delta_{m,r}$ sets of type $(I\setminus \{e\})c'$ with $c'\in (H\cap A')\setminus B$ and $\frac{\beta-(2^m-2^{m-r}-\sigma(\beta)+(-1)^{\beta+1}\cdot\delta_{m,r})}{2}$ sets of type $(I\setminus\{e\})c\cup(I\setminus\{e\})c^{-1}$ with $c\in H\setminus(A'\cup B)$, where $\delta$ is the Kronecker's delta.
			
			\item Suppose $B\leqslant H$, $m>|L|$ and $\beta\leq 2^m$. Let $S$ be a union of $\beta$ sets of type $(I\setminus(J\cup\{e\}))a'$ with $a'\in H\cap A'$ and $\beta/2$ sets of type $\{ha'',h^{-1}a''^{-1}\}$ with $ha''\in {\rm Nsq}(G)$ and $h\in H$ for each $Ha''\subseteq J$.
			
			\item Suppose $B\leqslant H$, $m>|L|$ and $\beta>2^m$. Let $S$ be a union of $2^m$ sets of type $(I\setminus(J\cup\{e\}))a'$ with $a'\in H\cap A'$, $\frac{\beta-2^m}{2}$ sets of type $(I\setminus(J\cup\{e\}))a\cup(I\setminus(J\cup\{e\}))a^{-1}$ with $a\in H\setminus A'$, and $\beta/2$ sets of type $\{ha'',h^{-1}a''^{-1}\}$ with $ha''\in {\rm Nsq}(G)$ and $h\in H$ for each $Ha''\subseteq J$.
			
			\item Suppose $B\nleqslant H$ and $m>|L|$ and $\beta\leq 2^m-2^{m-r}$. Let $S$ be a union of $\beta$ sets of type $(I\setminus(J\cup\{e\}))c'$ with $c'\in (H\cap A')\setminus B$ and $\beta/2$ sets of type $\{ha'',h^{-1}a''^{-1}\}$ with $ha''\in {\rm Nsq}(G)$ and $h\in H$ for each $Ha''\subseteq J$.
			
			\item Suppose $B\nleqslant H$ and $m>|L|$ and $\beta> 2^m-2^{m-r}$. Let $S$ be a union of $2^m-2^{m-r}$ sets of type $(I\setminus(J\cup\{e\}))c'$ with $c'\in (H\cap A')\setminus B$, $\frac{\beta-(2^m-2^{m-r})}{2}$ sets of type $(I\setminus(J\cup\{e\}))c\cup(I\setminus(J\cup\{e\}))c^{-1}$ with $c\in H\setminus(A'\cup B)$, and $\beta/2$ sets of type $\{ha'',h^{-1}a''^{-1}\}$ with $ha''\in {\rm Nsq}(G)$ and $h\in H$ for each $Ha''\subseteq J$.
		\end{itemize}
		
		We now claim that $H$ is a $(0,\beta)$-regular set in ${\rm CayS}(G,S)$. The proof proceeds in the following steps.
		
		\begin{stepp}\label{stepp-1}
			$S$ is normal.
		\end{stepp}
		Suppose $m=|L|$. By Lemma \ref{fact1}, we get $J=\emptyset$, and so $I\setminus\{e\}$ is inverse-closed, which imply that $S$ is normal from Corollary \ref{cor1}.
		
		Suppose $m>|L|$. In view of Lemma \ref{f6}, one obtains $J\neq\emptyset$. Since $I\setminus(J\cup\{e\})$ and $\{ha'',h^{-1}a''^{-1}\}$ with $ha''\in {\rm Nsq}(G)$ and $h\in H$ for each $Ha''\subseteq J$ are inverse-closed, from Corollary \ref{cor1}, $S$ is normal.
		
		\begin{stepp}\label{stepp-2}
			$S$ is square-free.
		\end{stepp}
		If $B\leqslant H$, from Lemma \ref{lem13} \ref{lem13-1}, then $S$ is square-free; if $B\nleqslant H$, from Lemma \ref{lem13} \ref{lem13-4}, then $S$ is square-free.
		
		\begin{stepp}\label{stepp-3}
			$|S\cap \langle H,zb\rangle x|=\beta$ for each $x\in A\setminus H$.
		\end{stepp}
		By \eqref{Aemptyset}, one gets $S\cap \langle H,zb\rangle x=S\cap Hx$ for each $x\in A\setminus H$. We only need to prove $|S\cap Hx|=\beta$ for each $x\in A\setminus H$.
		
		Suppose $m=|L|$. Note that $J=\emptyset$. Since $|(I\setminus\{e\})\cap Hx|=1$ for each $x\in A\setminus H$, we have $|S\cap Hx|=\beta$.
		
		Suppose $m>|L|$. By Lemma \ref{f6}, one has $J\neq\emptyset$. Since $|(I\setminus(J\cup\{e\}))\cap Hd|=1$ for each $Hd$ satisfying $Hd\cap J\cap H=\emptyset$, we have $|S\cap Hd|=\beta$. Since $S$ contains $\beta/2$ sets of type $\{ha'',h^{-1}a''^{-1}\}$ with $ha''\in {\rm Nsq}(G)$ and $h\in H$ for each $Ha''\subseteq J$, we have $|S\cap Ha''|=\beta$.
		
		In view of Steps \ref{step-1}--\ref{step-3} and Lemma \ref{lem3} \ref{lem3-1}, our claim is valid.
	\end{proof}

	\section{Proofs of Theorems of \ref{thm1} and \ref{thm2}}
	
	To give the proofs of Theorems of \ref{thm1} and \ref{thm2}, we need some auxiliary lemmas.

	\begin{lemma}\label{lem20}
		The subgroup $H$ is an $(\alpha,0)$-regular set of $G$ if and only if one of the following occurs:
		\begin{enumerate}
			\item\label{lem20-1} $0\leq\alpha\leq |H|/2-1$ and $\alpha$ is even, when $m=|L|=1$ and $b^2\in H\setminus B$;
			
			\item\label{lem20-2} $0\leq\alpha\leq (2^{|L|}-1)|H|/2^{|L|}-1$, when $m\geq|L|\geq 1$, $m\neq 1$ and $b^2\in H\setminus B$;
			
			\item\label{lem20-3} $0\leq\alpha\leq (2^{|L|}-1)|H|/2^{|L|}$ and $\alpha$ is even, when $r=0$ and $b^2\notin H\setminus B$;
			
			\item\label{lem20-4} $0\leq\alpha\leq (2^{|L|}-1)|H|/2^{|L|}$, when $r>0$ and $b^2\notin H\setminus B$.
		\end{enumerate}
	\end{lemma}
	\begin{proof}
		Let $H$ be an $(\alpha,0)$-regular set in ${\rm CayS}(G,S)$ for some subset $S$ of $G$. By Lemma \ref{lem3} \ref{lem3-1}, one obtains $S=S\cap H$. Since $b^2\in A'$ from Fact \ref{jb2}, we get
		\begin{align}\label{Salpha0}
			S=S\cap H &\subseteq {\rm Nsq}(G)\cap H\nonumber\\
			&=({\rm Nsq}(G)\cap (H\cap A'))\cup ({\rm Nsq}(G)\cap (H\setminus A'))\nonumber\\
			&=((H\cap A')\setminus (B\cup\{b^2\}))\cup (H\setminus (A'\cup B)).
		\end{align}
		
		We divide the proof into the following two cases.
		
		\textbf{Case 1.} $b^2\in H\setminus B$.
		
		By Fact \ref{jb2}, one gets ${\rm Nsq}(G)\cap H=H\setminus((H\cap B)\cup\{b^2\})$. If $|L|=0$, from Lemma \ref{fact2} \ref{fact2-1}, then $|H\cap B|=|H|$, and so $H\leqslant B$, contrary to $b^2\in H\setminus B$. Then $|L|\geq 1$. In view of Lemma \ref{lem3} \ref{lem3-2} and Lemma \ref{fact2} \ref{fact2-1}, we have $0\leq\alpha\leq |H\setminus((H\cap B)\cup\{b^2\})|=(2^{|L|}-1)|H|/2^{|L|}-1$.
		
		Suppose $m=1$. It follows that $m=|L|=1$. By Lemma \ref{fact2} \ref{fact2-1}, we have $|H\cap A'|=2$. Since $\langle b^2\rangle\leqslant H\cap A'$, we get $H\cap A'=\langle b^2\rangle$, and so $H\cap A'\subseteq {\rm Sq}(G)$ from Fact \ref{jb2}. Since $S$ is square-free and $S=S\cap H$, we have $S\subseteq {\rm Nsq}(G)\cap H=H\setminus{\rm Sq}(G)$. Since $H\cap A'\subseteq {\rm Sq}(G)$, from Fact \ref{jb3}, $S$ has no involution. Since $S$ is normal, from Lemma \ref{lem2}, $S$ is a union of $\alpha/2$ sets of type $\{a,a^{-1}\}$ with $a\in H\setminus((H\cap B)\cup\{b^2\})$. It follows that $\alpha$ is even. Thus, \ref{lem20-1} holds.
		
		Suppose $m\neq 1$. Then $m\geq|L|\geq 1$ and $m\neq 1$. By Lemma \ref{fact2} \ref{fact2-1}, we have $|(H\cap A')\setminus B|=2^m-2^{m-r}$. Since $b^2\in H\setminus B$ and $b^2\in A'$, from Fact \ref{jb3}, one gets $|(H\cap A')\setminus (B\cup\{b^2\})|=2^m-2^{m-r}-1$, which implies that $S$ has at most $2^m-2^{m-r}-1$ involutions from \eqref{Salpha0}. If $\alpha\leq 2^m-2^{m-r}-1$, from Lemma \ref{lem2} and \eqref{Salpha0}, then $S$ can be a union of $\alpha$ sets of type $\{a'\}$ with $a'\in (H\cap A')\setminus (B\cup\{b^2\})$ since $S$ is normal.
		
		We only need to consider the case $\alpha>2^m-2^{m-r}-1$. Note that $S$ is normal. If $\alpha$ is even, from Lemma \ref{lem2} and \eqref{Salpha0}, then $S$ can be a union of $2^m-2^{m-r}-2+\delta_{m,r}$ sets of type $\{a'\}$ with $a'\in (H\cap A')\setminus (B\cup\{b^2\})$ and $\frac{\alpha-(2^m-2^{m-r}-2+\delta_{m,r})}{2}$ sets of type $\{a,a^{-1}\}$ with $a\in H\setminus (A'\cup B)$. If $\alpha$ is odd, from Lemma \ref{lem2} and \eqref{Salpha0} again, then $S$ can be a union of $2^m-2^{m-r}-1-\delta_{m,r}$ sets of type $\{a'\}$ with $a'\in (H\cap A')\setminus (B\cup\{b^2\})$ and $\frac{\alpha-(2^m-2^{m-r}-1-\delta_{m,r})}{2}$ sets of type $\{a,a^{-1}\}$ with $a\in H\setminus (A'\cup B)$. Thus, \ref{lem20-2} is valid.
		
		\textbf{Case 2.} $b^2\notin H\setminus B$.
		
		Since $b^2\notin H\setminus B$, from Fact \ref{jb2}, we have ${\rm Nsq}(G)\cap H=H\setminus(H\cap B)$. It follows from Lemma \ref{lem3} \ref{lem3-2} and Lemma \ref{fact2} \ref{fact2-1} that $0\leq\alpha\leq |H\setminus(H\cap B)|=(2^{|L|}-1)|H|/2^{|L|}$.
		
		Suppose $r=0$. By Lemma \ref{fact2} \ref{fact2-1}, one has $|H\cap A'|=|H\cap A'\cap B|=2^m$, and so $H\cap A'\subseteq B$. Since $S$ is square-free and $S=S\cap H$, we have $S\subseteq {\rm Nsq}(G)\cap H=H\setminus(H\cap B)$. Since $H\cap A'\subseteq B$, from Fact \ref{jb3}, $S$ has no involution. Since $S$ is normal, from Lemma \ref{lem2}, $S$ is a union of $\alpha/2$ sets of type $\{a,a^{-1}\}$ with $a\in H\setminus (H\cap B)$. It follows that $\alpha$ is even. Thus, \ref{lem20-3} holds.
		
		Suppose $r>0$. Since $b^2\notin H\setminus B$, we get $(H\cap A')\setminus (B\cup\{b^2\})=(H\cap A')\setminus B$. By Lemma \ref{fact2} \ref{fact2-1}, we get $|(H\cap A')\setminus B|=2^m-2^{m-r}$. By \eqref{Salpha0} and Fact \ref{jb3}, $S$ has at most $2^m-2^{m-r}$ involutions. If $\alpha \leq 2^m-2^{m-r}$, from Lemma \ref{lem2} and \eqref{Salpha0}, then $S$ can be a union of $\alpha$ sets of type $\{a'\}$ with $a'\in (H\cap A')\setminus B$ since $S$ is normal.
		
		We only need to consider the case $\alpha>2^m-2^{m-r}$. Note that $S$ is normal. If $\alpha$ is even, from Lemma \ref{lem2} and \eqref{Salpha0}, then $S$ can be a union of $2^m-2^{m-r}-\delta_{m,r}$ sets of type $\{a'\}$ with $a'\in (H\cap A')\setminus B$ and $\frac{\alpha-(2^m-2^{m-r}-\delta_{m,r})}{2}$ sets of type $\{a,a^{-1}\}$ with $a\in H\setminus (A'\cup B)$. If $\alpha$ is odd, from Lemma \ref{lem2} and \eqref{Salpha0} again, then $S$ can be a union of $2^m-2^{m-r}-1+\delta_{m,r}$ sets of type $\{a'\}$ with $a'\in (H\cap A')\setminus B$ and $\frac{\alpha-(2^m-2^{m-r}-1+\delta_{m,r})}{2}$ sets of type $\{a,a^{-1}\}$ with $a\in H\setminus (A'\cup B)$. Thus, \ref{lem20-4} is valid.
	\end{proof}

	\begin{lemma}\label{lem16} The subgroup $\langle H,zb\rangle$ is an $(\alpha,\beta)$-regular set of $G$ if and only if there exists a normal and square-free subset $S$ of $G$ such that $\langle H,zb\rangle$ is an $(\eta,\zeta)$-regular set in ${\rm CayS}(G,S\cap A)$ and an $(\alpha-\eta,\beta-\zeta)$-regular set in ${\rm CayS}(G,S\cap Ab)$ for some $\eta\in\{0,1,\ldots,\alpha\}$ and $\zeta\in\{0,1,\ldots,\beta\}$.
	\end{lemma}
	\begin{proof}
		We first prove the necessity. Suppose that $\langle H,zb\rangle$ is an $(\alpha,\beta)$-regular set in ${\rm CayS}(G,S)$.  By Lemma \ref{lem3} \ref{lem3-1}, one gets $|S\cap \langle H,zb\rangle|=\alpha$ and $|S\cap \langle H,zb\rangle x|=\beta$ for each $x\in A\setminus H$, which imply that $S$ is the union of $S\cap \langle H,zb\rangle$, and $\beta$ pairwise disjoint subsets $T_j$ of ${\rm Nsq}(G)$, $1\leq j\leq\beta$, such that for each $j$, $T_{j}\cup\{e\}$ is a right transversal of $\langle H,zb\rangle$ in $G$. Without loss of generality, we may assume that $|S\cap H|=\eta$, $\cup_{h=1}^{\zeta}T_h\subseteq A$ and $\cup_{h=\zeta+1}^{\beta}T_h\subseteq Ab$ for some $\eta\in\{0,1,\ldots,\alpha\}$ and $\zeta\in\{0,1,\ldots,\beta\}$. Since $H$ has index $2$ of $\langle H,zb\rangle$, we have $|S\cap Hzb|=\alpha-\eta$. Then $|S\cap A\cap \langle H,zb\rangle|=\eta$, $|S\cap A\cap \langle H,zb \rangle x|=\zeta$ for each $x\in A\setminus H$ and $|S\cap Ab\cap \langle H,zb\rangle|=\alpha-\eta$,  $|S\cap Ab\cap \langle H,zb \rangle x|=\beta-\zeta$ for each $x\in A\setminus H$. Note that $S=(S\cap A)\cup(S\cap Ab)$. Since $S$ is a normal and square-free subset of $G$, $S\cap A$ and $S\cap Ab$ are both normal and square-free. By Lemma \ref{lem3} \ref{lem3-1}, $\langle H,zb\rangle$ is an $(\eta,\zeta)$-regular set in ${\rm CayS}(G,S\cap A)$ and an $(\alpha-\eta,\beta-\zeta)$-regular set in ${\rm CayS}(G,S\cap Ab)$.
		
		Next we prove the sufficiency. Suppose that $\langle H,zb\rangle$ is an $(\eta,\zeta)$-regular set in ${\rm CayS}(G,S\cap A)$ and an $(\alpha-\eta,\beta-\zeta)$-regular set in ${\rm CayS}(G,S\cap Ab)$. By Lemma \ref{lem3} \ref{lem3-1}, $S\cap A$ contains exactly $\eta$ elements of $\langle H,zb\rangle$ and $\zeta$ elements of $\langle H,zb\rangle x$ for each $x\in A\setminus H$, and $S\cap Ab$ contains exactly $\alpha-\eta$ elements of $\langle H,zb\rangle$ and $\beta-\zeta$ elements of $\langle H,zb\rangle x$ for each $x\in A\setminus H$.
		Since $G=A\cup Ab$, one gets $S=(S\cap A)\cup (S\cap Ab)$, which implies that $S$ contains exactly $\alpha$ elements of $\langle H,zb\rangle$ and $\beta$ elements of $\langle H,zb\rangle x$ for each $x\in A\setminus H$. Since $S\cap A$ and $S\cap Ab$ are both normal and square-free, from Lemma \ref{lem3} \ref{lem3-1}, $\langle H,zb\rangle$ is an $(\alpha,\beta)$-regular set in ${\rm CayS}(G,S)$.
	\end{proof}
	
	\begin{lemma}\label{lem5}
		Let $S$ be a normal and square-free subset of $G$ with $S\subseteq A$. Then $\langle H,zb\rangle$ is an $(\alpha,\beta)$-regular set in ${\rm CayS}(G,S)$ if and only if $H$ is an $(\alpha,0)$-regular set in ${\rm CayS}(G,S\cap H)$ and $\langle H,zb\rangle$ is a $(0,\beta)$-regular set of ${\rm CayS}(G,S\setminus H)$.
	\end{lemma}
	\begin{proof}
		Since $S\subseteq A$ and $H$ has index $2$ of $\langle H,zb\rangle$, we have $S\cap \langle H,zb\rangle=S\cap H$ and $S\cap (G\setminus \langle H,zb\rangle)=S\cap (A\setminus H)$.
		
		We first prove the necessity. Since $S$ is normal, square-free and $H$ is a normal subgroup of $G$, $S\cap H$ is normal and square-free, which implies that $S\setminus H$ are normal and square-free. Since $\langle H,zb\rangle$ is an $(\alpha,\beta)$-regular set in ${\rm CayS}(G,S)$, from Lemma \ref{lem3} \ref{lem3-1}, one has $|S\cap H|=|S\cap \langle H,zb\rangle|=\alpha$ and $|(S\setminus H)\cap \langle H,zb\rangle x|=|S\cap \langle H,zb\rangle x|=\beta$ for each $x\in A\setminus H$. Since $S\cap Hb=\emptyset$, from Lemma \ref{lem3} \ref{lem3-1} again, $H$ is an $(\alpha,0)$-regular set in ${\rm CayS}(G,S\cap H)$ and $\langle H,zb\rangle$ is a $(0,\beta)$-regular set in ${\rm CayS}(G,S\setminus H)$.
		
		We next prove the sufficiency. Since $H$ is an $(\alpha,0)$-regular set in ${\rm CayS}(G,S\cap H)$ and $\langle H,zb\rangle$ is a $(0,\beta)$-regular set in ${\rm CayS}(G,S\setminus H)$, from Lemma \ref{lem3} \ref{lem3-1}, one gets $|S\cap \langle H,zb\rangle|=|S\cap H|=\alpha$ and $|S\cap \langle H,zb\rangle x|=|(S\setminus H)\cap \langle H,zb\rangle x|=\beta$ for each $x\in A\setminus H$. Since $S=(S\cap H)\cup (S\setminus H)$, from Lemma \ref{lem3} \ref{lem3-1}, $\langle H,zb\rangle$ is an $(\alpha,\beta)$-regular set in ${\rm CayS}(G,S)$.
	\end{proof}
	
	Now we are ready to give a proof of Theorem \ref{thm1}.
	
	\begin{proof}[Proof of Theorem \ref{thm1}] By Lemma \ref{lem4}, $H$ is an $(\alpha,\beta)$-regular set of $G$ if and only if $H$ is an $(\alpha,0)$-regular set of $G$ and a $(0,\beta)$-regular set of $G$. The desired result is valid from Lemmas \ref{lem8} and \ref{lem20}.
	\end{proof}
	
	Next, we give a proof of Theorem \ref{thm2}.
	
	\begin{proof}[Proof of Theorem \ref{thm2}] 
		
		By Lemma \ref{lem16}, $\langle H,zb\rangle$ is an $(\alpha,\beta)$-regular set of $G$ if and only if $\langle H,zb\rangle$ is an $(\eta,\zeta)$-regular set in ${\rm CayS}(G,S\cap A)$ and an $(\alpha-\eta,\beta-\zeta)$-regular set in ${\rm CayS}(G,S\cap Ab)$ for some $\eta\in\{0,1,\ldots,\alpha\}$, $\zeta\in\{0,1,\ldots,\beta\}$, and normal, square-free subset $S$. In view of Lemma \ref{lem5}, $\langle H,zb\rangle$ is an $(\alpha,\beta)$-regular set of $G$ if and only if $H$ is an $(\eta,0)$-regular set in ${\rm CayS}(G,S\cap H)$, $\langle H,zb\rangle$ is an $(0,\zeta)$-regular set in ${\rm CayS}(G,(S\cap A)\setminus H)$ and an $(\alpha-\eta,\beta-\zeta)$-regular set in ${\rm CayS}(G,S\cap Ab)$ for some $\eta\in\{0,1,\ldots,\alpha\}$, $\zeta\in\{0,1,\ldots,\beta\}$, and normal, square-free subset $S$. The desired result follows from Lemma \ref{fact2} \ref{fact2-2} and Lemmas \ref{lem9}, \ref{lem18} and \ref{lem20}.
	\end{proof}
	
	
\section*{Data Availability Statement}
	
No data was used for the research described in the article.

\section*{Conflict of Interest}
	
The authors declare that they have no conflict of interest.

	\section*{Acknowledgements}
	
The authors would like to thank Dr. Wenying Zhu for her help and valuable suggestions Y.~Yang is supported by NSFC (12101575, 52377162).

\end{document}